\documentclass[11pt, a4paper]{amsart}
\usepackage{amssymb}
\usepackage{amsthm}
\usepackage{amsmath}
\usepackage{graphicx}
\usepackage[a4paper, total={170mm,257mm}, left=20mm, top=20mm]{geometry}
\usepackage{enumitem}
\setlist[enumerate, 1]{label={\textup{(\roman*)}}}
\setlist[enumerate, 2]{label={\textup{(\alph*)}}}
\setlist[enumerate]{leftmargin=1em, labelsep=.5em, labelwidth=1.5em, itemindent=2em}
\usepackage[unicode=true, breaklinks=true]{hyperref}
\hypersetup{
    pagebackref  = true,
	colorlinks   = true,
	pdfauthor    = {Matěj Doležálek},
	pdftitle     = {Quaternions and universal quadratic forms over number fields},
	pdfkeywords  = {},
	pdfcreator   = {Pdflatex},
	pdfpagemode  = UseNone,
	pdfstartview = FitH
}

\newtheorem{theorem}{Theorem}
\newtheorem{lemma}[theorem]{Lemma}
\newtheorem{corollary}[theorem]{Corollary}
\theoremstyle{definition}
\newtheorem{definition}[theorem]{Definition}

\let\phi\varphi
\newcommand{\Z}{\mathbb{Z}}
\newcommand{\Q}{\mathbb{Q}}
\newcommand{\R}{\mathbb{R}}
\newcommand{\inv}[1]{#1^{-1}}
\newcommand{\nequiv}{\not\equiv}
\newcommand{\ceil}[1]{\left\lceil#1\right\rceil}
\newcommand{\abs}[1]{\left\lvert#1\right\rvert}
\newcommand{\zav}[1]{\left(#1\right)}
\newcommand{\set}[1]{\left\{#1\right\}}
\newcommand{\dd}{\,\mathrm{d}}
\newcommand{\matr}{\operatorname{M}}
\newcommand{\vol}{\operatorname{Vol}}
\newcommand{\nm}{\operatorname{Nm}}
\newcommand{\nmk}{\nm_{K}}
\newcommand{\nmh}{\nm_{\mathbb{H}}}
\newcommand{\nn}{\operatorname{NN}}
\newcommand{\h}{\mathcal{H}}
\newcommand{\mtrx}[1]{\left(\begin{matrix}#1\end{matrix}\right)}
\newcommand{\ok}{\mathcal{O}_K}

\title{Quaternions and universal quadratic forms over number fields}
\author{Mat\v ej Dole\v z\'alek}
\address{Charles University, Faculty of Mathematics and Physics, Department of Algebra, Sokolov\-sk\' a 83, 18600 Praha~8, Czech Republic}
\email{matej.dolezalek.271828@gmail.com}

\thanks{The author was supported by student faculty grant of the Faculty of Mathematics and Physics of Charles University.}

\date{\today}
\keywords{universal quadratic form, quaternions, totally real number field, geometry of numbers}

\begin{document}

\begin{abstract}
    We study quadratic forms over totally real number fields by using an associated ring of quaternions. We examine some properties of residue class rings of these quaternions and use geometry of numbers to prove that certain ideals of the ring of quaternions contain elements of a small norm. We prove that $x^2+y^2+z^2+w^2+xy+xz+xw$ is universal over $\Q(\zeta_7+\zeta_7^{-1})$ and that $x^2+xy+y^2+z^2+zw+w^2$ represents all totally positive multiples of certain special elements.
\end{abstract}
\maketitle

\section{Introduction}
\label{sec:intro}
The problem of representing positive integers with a given quadratic form has been given considerable attention by number theorists throughout the centuries. The classical results include the famous four-square theorems of Lagrange and Jacobi, which state that the quadratic form $x^2+y^2+z^2+w^2$ represents any positive integer $n$ and that the number of such representations is $8\sum_{4\nmid d\mid n}d$. While Lagrange's original proof and used infinite descent and Jacobi's utilized elliptic functions, Hurwitz \cite{hurwitz} later showed that both can be proved by studying the algebraic properties of certain rings of quaternions and other proofs of Lagrange's four-square theorem use Minkowski's theorem on a convex body and a lattice.

The study of universal quadratic forms extends beyond rational numbers and rational integers to number fields $K$ and their rings of algebraic integers $\ok$.
The existence of a universal form in a small number of variables has been studied extensively and seems rare \cite{blomer-kala1, blomer-kala2, cech-etal, chan-kim-raghavan, earnest-khosravani, kala, BMkim, MHkim, kala-svoboda, krasensky-tinkova-zemkova, yatsyna}.
In this article, we will mainly concern ourselves with the case of a totally real field $K$ whose $\ok$ is a unique factorization domain.
When we look specifically at forms in four variables, these can often be advantageously expressed using quaternions over the respective number field.
Deutsch has used the methods of quaternions and geometry of numbers to prove some results on universality and representation by quaternary quadratic forms over rational numbers and various real quadratic fields \cite{deutsch1, deutsch2, deutsch3, deutsch4, deutsch5, deutsch6, deutsch7}. Kala and Yatsyna \cite{kala-yatsyna} studied the existence of a universal quadratic forms with (rational) integer coefficients and proved that amongst certain fields of small degree, the only ones that admit a universal quadratic form with integer coefficients are $\Q$, $\Q(\sqrt5)$ and $\Q(\zeta_7+\zeta_7^{-1})$. In particular, they proved that
\begin{equation}
    \label{eq:hurwitz-form}
    x^2+y^2+z^2+w^2 + xy+xz+xw
\end{equation}
is universal over $\Q(\zeta_7+\zeta_7^{-1})$ and conjectured that
\begin{equation}
    \label{eq:conjecture-form}
    x^2+xy+y^2 + z^2+zw+w^2
\end{equation}
is also universal. In this article, we will give a proof of the universality of \eqref{eq:hurwitz-form} that utilizes quaternions and geometry of numbers.
We will also use the same methods to prove a weaker result about representation of elements of $\ok$ by \eqref{eq:conjecture-form}.

The method of quaternionic rings is based on expressing a given quadratic form as the norm of elements of some ring of quaternions. In many cases, the universality of a form can then be proven based on such a ring being a principal ideal domain. Thus, the question of representation by a quadratic forms with four variables is linked to the class number of the associated quaternionic ring. The rings corresponding to forms \eqref{eq:hurwitz-form} and \eqref{eq:conjecture-form} over $\Q(\zeta_7+\zeta_7^{-1})$ are known to have class numbers $1$ and $2$ respectively \cite[Tables 8.2, 8.3]{kirschmer-voight}. In case of \eqref{eq:hurwitz-form}, we will give a geometric proof of this.

In Sections~\ref{sec:quaternions}--\ref{sec:gon} we shall work with a totally real number field $K$ whose ring of integers is a unique factorization domain. The main result of these sections will be Theorem~\ref{thrm:suitabilitybound}, which provides a bound on the norm of denominators of fractional ideals of a certain ring of quaternions that need to be checked (see Section~\ref{sec:quaternions} for a precise definition) in order to prove that this ring of quaternions is a principal ideal domain. The bound of Theorem~\ref{thrm:suitabilitybound} and its proof are similar to Minkowski's bound for a number field.

In Section~\ref{sec:pid}, we shall additionally require that $\ok$ contain a unit of every signature, which will allow us to formulate a sufficient condition for a certain type of a quadratic form in four variables (see Definition~\ref{def:H}) to be universal over $K$.

In Section~\ref{sec:2form}, we will examine the cubic field $K=\Q(\zeta_7+\zeta_7^{-1})$ and show that a quadratic form equivalent to \eqref{eq:hurwitz-form} is universal over $K$, meaning that \eqref{eq:hurwitz-form} is also universal. We will also give an explicit formula for the number of representations of a given element, akin to Jacobi's four-square theorem. Finally, in Section~\ref{sec:3form}, we will show that \eqref{eq:conjecture-form} represents all totally positive multiples of certain special elements.

\section{Quaternions over totally real number fields}
\label{sec:quaternions}
Throughout this article, let $K$ be a totally real number field of degree $n$ over $\Q$ and discriminant $d_K$, with $n$ distinct real embeddings $\sigma_1,\sigma_2,\dots,\sigma_n: K\hookrightarrow \R$. Let us also presume that the ring $\ok$ of algebraic integers of $K$ is a unique factorization domain, i.e.\ that the class number of $K$ is $1$. We define the norm of an element $\lambda\in K$ as $\nmk(\lambda)=\sigma_1(\lambda)\cdots \sigma_n(\lambda)$. For $\lambda,\lambda'\in K$, we use $\lambda \succ \lambda'$ to mean $\sigma_t(\lambda)>\sigma_t(\lambda')$ for all $t\in\set{1,\dots,n}$ and we say that $\lambda$ is \emph{totally positive}, if $\lambda\succ 0$. Further, let $K^+$ and $\ok^+$ denote the sets of all totally positive elements of $K$ and $\ok$ respectively.

Given a quadratic form $Q(x_1,\dots,x_d)$ with coefficients from $\ok$, we say that it is \emph{totally positive definite} over $K$, if $Q(x_1,\dots,x_d) \succ 0$ for any $x_1,\dots,x_d\in K$ that are not all zero. For $\lambda\in\ok$, we say that $Q$ \emph{represents} $\lambda$, if the equation $Q(x_1,\dots,x_d)=\lambda$ has a solution $x_1,\dots,x_d\in\ok$. Finally, we say that a totally positive definite quadratic form is \emph{universal} over $K$, if it represents all elements of $\ok^+$.

Let $i$, $j$, $k$ by the usual basis elements of quaternions satisfying the equation
\[
    i^2=j^2=k^2=ijk=-1.
\]
For any quaternion $L=x+yi+zj+wk$, $x,y,z,w\in\R$, we define its \emph{conjugate} $\overline{L} = x-yi-zj-wk$ and its \emph{norm} $\nmh(L) = L\overline{L}=x^2+y^2+z^2+w^2$. We distinguish between this quaternionic norm and the usual field norm of $K$; we will also sometimes consider their composition, which we will call the \emph{double norm} and denote it $\nn = \nmk \circ \nmh$.

Let us now construct a family of rings of quaternions over $K$. The notation of the following definition will stay fixed throughout this article:
\begin{definition}
    \label{def:H}
    Let us choose $A,B,\mu,\nu\in\ok$ and set
    \begin{align*}
        S &= 4A-\mu^2, & T &= BS-\nu^2
    \end{align*}
    in such a way that both $S$ and $T$ are totally positive. Further, set
    \begin{align*}
        \alpha &= \frac{\mu+i\sqrt{S}}{2}, \qquad \beta = \frac{\nu i+j\sqrt{T}}{\sqrt{S}},\\
        \h &= \set{x+y\alpha+z\beta+w\alpha\beta : x,y,z,w\in \ok},\\
        Q(x,y,z,w) &= \zav{x^2+\mu xy+Ay^2} + \nu(yz-xw) + B\zav{z^2+\mu zw+Aw^2}.
    \end{align*}
    Also let $\sigma_t(\alpha)$, $\sigma_t(\beta)$, $\sigma_t(\h)$ and $\sigma_t(Q)$ denote the quaternions, ring and quadratic form that we would get by choosing $\sigma_t(A)$ instead of $A$, $\sigma_t(B)$ instead of $B$ etc.
\end{definition}

\begin{lemma}
    The following holds for any choice of $A$, $B$, $\mu$ and $\nu$ satisfying the constraints of Definition~\ref{def:H}:
    \begin{enumerate}
        \item $\h$ forms a ring (a domain in fact) with quaternion multiplication and addition.
        \item $Q(x,y,z,w)=\nmh(x+y\alpha+z\beta+w\alpha\beta)$ for any $x,y,z,w\in \ok$.
        \item The transformation
        \begin{equation}
            \label{eq:to4squares}
            \mtrx{
            X\\Y\\Z\\W
            }
            =
            \mtrx{
            1&\frac{\mu}{2}&0&-\frac{\nu}{2}\\
            0&\frac{\sqrt{S}}{2}&\frac{\nu}{\sqrt{S}}&\frac{\mu\nu}{2\sqrt{S}}\\
            0&0&\sqrt{\frac{T}{S}}&\frac{\mu}{2}\sqrt{\frac{T}{S}}\\
            0&0&0&\frac{\sqrt{T}}{2}
            }
            \mtrx{
            x\\y\\z\\w
            }
        \end{equation}
        turns $Q$ into the sum of four squares $X^2+Y^2+Z^2+W^2$.
    \end{enumerate}
\end{lemma}
\begin{proof}
    All of these are easy to verify directly.
\end{proof}
\begin{corollary}
    $Q$ is totally positive definite over $K$.
\end{corollary}
\begin{proof}
    Norm of any quaternion different from $0$ is positive and $x+y\alpha+z\beta+w\alpha\beta=0$, if and only if $x=y=z=w=0$. Further, we have
    \[
        \sigma_t\Big(Q(x,y,z,w)\Big) = (\sigma_t(Q))\Big(\sigma_t(x), \sigma_t(y), \sigma_t(z), \sigma_t(w)\Big)
    \]
    and $\sigma_t(Q)$ must, by the same argument, attain only positive values unless all of its arguments are zero.
\end{proof}

In the remainder of this section, we will prove some results about the existence of elements of $\h$ with certain useful properties.
\begin{lemma}
    \label{lem:modvalues}
    Let $\rho \in \ok$ be a prime element and consider $a,b,c\in\ok$ such that at least one of $a$, $b$ is not divisible by $\rho$. Then the polynomial $ax^2+bx+c$ attains at least $\frac{\abs{\nmk(\rho)}}{2}$ different values mod $\rho$.
\end{lemma}
\begin{proof}
    For any $d\in\ok$, $\rho\nmid a$ or $\rho\nmid b$ means that the polynomial $ax^2+bx+c-d$ is non-constant over the finite field $\ok/\rho\ok$, so it has at most two roots in it. Since any $x\in\ok/\rho\ok$ is a root of one such polynomial, there need to be at least $\frac{\abs{\ok/\rho\ok}}{2} = \frac{\abs{\nmk(\rho)}}{2}$ distinct values of the polynomial $ax^2+bx+c$ mod $\rho$.
\end{proof}

\begin{lemma}
    \label{lem:precursors}
    Let $\rho\in\ok$ be a prime element that does not divide $T$. Then there are $x,z\in\ok$ satisfying
    \[
        \rho\mid \nmh(x+\alpha+z\beta).
    \]
\end{lemma}
\begin{proof}
    Let $L=x+\alpha+z\beta$ for variables $x,z\in\ok$. Let us distinguish two cases:
    \begin{enumerate}
        \item $\rho\mid 2$, so $\nmk(\rho)$ is even. We will presume that
        \begin{equation}
            \label{eq:rhodivides2}
            x^2+\mu x+A+\nu z+Bz^2 \equiv 0 \pmod{\rho}
        \end{equation}
        has no solution and arrive at a contradiction. Since $\nmk(\rho)$ is even, the multiplicative group of the finite field $\ok/\rho\ok$ has an odd order, and since it is cyclic, the mapping $x\mapsto x^2$ must be bijective on it. Because also $0^2=0$, this means that any element of $\ok/\rho\ok$ can be expressed as a square, so we may write $B\equiv b^2 \pmod{\rho}$ for some $b$. With this, the polynomial in \eqref{eq:rhodivides2} can be rewritten as
        \[
            (x+bz)^2 + \mu(x+bz)+A + z(b\mu+\nu).
        \]
        If $b\mu+\nu\nequiv 0\pmod{\rho}$, then \eqref{eq:rhodivides2} clearly has a solution, so it must be the case that $b\mu+\nu\equiv 0\pmod{\rho}$. But this means that
        \[
            T = 4AB-B\mu^2-\nu^2 \equiv B\mu^2+\nu^2 \equiv (b\mu+\nu)^2 \equiv 0 \pmod{\rho},
        \]
        which is a contradiction. So \eqref{eq:rhodivides2} has a solution.

        \item Suppose that $2\nmid\rho$, so $r:=\abs{\nmk(\rho)}$ is odd. We have $\rho\nmid 1$, so by Lemma~\ref{lem:modvalues}, the expression $x^2+x+A$ attains at least $\ceil{\frac{r}{2}}=\frac{r+1}{2}$ different values mod $\rho$. Similarly, $\rho\mid B,\nu$ would mean $\rho\mid T$, so the aforementioned lemma means that $-Bz^2-\nu z$ also attains at least $\frac{r+1}{2}$ different values. Since
        \[
            \frac{r+1}{2} + \frac{r+1}{2} = r+1 > r = \abs{\ok/\rho\ok},
        \]
        by the pigeonhole principle, the congruence
        \[
            x^2+x+A \equiv -Bz^2-\nu z\pmod{\rho}
        \]
        holds for some $x,z\in\ok$, as we wished to prove. \qedhere
    \end{enumerate}
\end{proof}

\section{Residue class rings of $\h$}
\label{sec:orbits}

\def\modop#1{\ \text{mod}\ #1}
In this section, we will examine some properties of residue rings $\h/\rho\h$, i.e.\ quaternions modulo some (prime) element $\rho\in\ok$. Clearly if $\rho\mid L_1-L_2$, then also $\rho = \overline\rho \mid \overline{L_1}-\overline{L_2}$. This means that the conjugate $\overline{L}$ is well-defined for $L\in\h/\rho\h$, so the norm $\nmh(L)=L\cdot\overline{L}\in\ok/\rho\ok$ is also well-defined. In other words, $\nmh(L\modop\rho) = \nmh(L)\modop\rho$.

\begin{lemma}
    \label{lem:matrixisomorphism}
    For a prime element $\rho\nmid T$, there exists an isomorphism $\psi$ from $\h/\rho\h$ to the ring $\matr_2(\ok/\rho\ok)$ of $2\times2$ matrices over $\ok/\rho\ok$ such that
    \begin{equation}
        \label{eq:isomorphism}
        \nmh(L)\equiv\det\psi(L)\pmod{\rho}
    \end{equation}
    for all $L\in\h/\rho\h$.
\end{lemma}
\begin{proof}
    By Lemma~\ref{lem:precursors}, there exist $e,f\in \ok$ such that $\rho\mid\zav{e^2+\mu e+A+\nu f+Bf^2}$. For $L=x+y\alpha+z\beta+w\alpha\beta\in\h/\rho\h$, we set $\psi(L)=\left(\begin{matrix}X&Y\\Z&W\end{matrix}\right)$, where
    \[
        \left(\begin{matrix}
        X\\Y\\Z\\W
        \end{matrix}\right)
        =
        \left(\begin{matrix}
        1&e+\mu&0&Bf\\
        0&-f&1&e+\mu\\
        0&-(Bf+\nu)&-B&Be\\
        1&-e&0&-(Bf+\nu)
        \end{matrix}\right)
        \left(\begin{matrix}
        x\\y\\z\\w
        \end{matrix}\right).
    \]
    This mapping is linear and its determinant is $T\nequiv 0\pmod{\rho}$, so it is injective. Thus it is also bijective, since $\abs{\h/\rho\h}=\abs{\ok/\rho\ok}^4 = \abs{\matr_2(\ok/\rho\ok)}$. With \[A\equiv -e^2-\mu e - Bf^2-\nu f\pmod{\rho},\] we can directly verify that $\psi(L_1L_2)=\psi(L_1)\psi(L_2)$ for any two basis elements $L_1,L_2\in\set{1,\alpha,\beta,\alpha\beta}$, which implies that $\psi(L_1L_2)=\psi(L_1)\psi(L_2)$ for any $L_1,L_2\in\h/\rho\h$ through linearity. Similarly, \eqref{eq:isomorphism} can be directly verified.
\end{proof}

\begin{definition}
    Let $\rho\in \ok$ be a prime. For $L\in(\h/\rho\h)\setminus\set{0}$ such that $\rho\mid \nmh(L)$, we define its (left) \emph{$\rho$-orbit} as the set
    \(
        \set{DL : D\in\h/\rho\h}
    \).
\end{definition}

\begin{lemma}
    \label{lem:orbits}
    Let $\rho\in \ok$ be a prime element that does not divide $T$ and has norm $r=\abs{\nmk(\rho)}$.
    We claim that:
    \begin{enumerate}
        \item Any $\rho$-orbit has exactly $r^2$ elements,
        \item The intersection of any two different $\rho$-orbits is $\set{0}$,
        \item There are exactly $r+1$ different $\rho$-orbits.
    \end{enumerate}
\end{lemma}
\begin{proof}
    \begin{enumerate}
        \item Let us consider when $D_1L\equiv D_2L\pmod{\rho}$, or equivalently when $DL\equiv 0\pmod{\rho}$ (where $D=D_1-D_2$). We shall use the isomorphism $\psi$ from Lemma~\ref{lem:matrixisomorphism}. Let
        \begin{align*}
            \psi(L) &=
            \left(\begin{matrix}
                x&y\\
                z&w
            \end{matrix}\right),
            &
            \psi(D) &=
            \left(\begin{matrix}
                a&b\\
                c&d
            \end{matrix}\right).
        \end{align*}
        We have $\rho\mid\nmh(L)$, which means $xw\equiv yz\pmod{\rho}$, and since $L\nequiv 0\pmod{\rho}$, the matrix $\psi(L)$ has a non-zero entry -- without loss of generality let it be $x$. The congruence $DL\equiv0\pmod{\rho}$ is then equivalent to the system
        \begin{align*}
            ax+bz &\equiv 0\pmod{\rho},\\
            ay+bw &\equiv 0\pmod{\rho},\\
            cx+dz &\equiv 0\pmod{\rho},\\
            cy+dw &\equiv 0\pmod{\rho}.
        \end{align*}
        The first and third congruence then imply $a\equiv -bz\inv{x}$ and $c\equiv -dz\inv{x}$ respectively. Because the determinant of $\psi(L)$ is zero, this then also ensures
        \begin{align*}
            ay+bw &\equiv -byz\inv{x}+bw \equiv b\inv{x}(xw-yz) \equiv 0\pmod{\rho},\\
            cy+dw &\equiv -dyz\inv{x}+dw \equiv d\inv{x}(xw-yz) \equiv 0\pmod{\rho}.
        \end{align*}
        This means that any suitable $\psi(D)$ uniquely corresponds to any choice of $b$ and $d$, of which there are $r^2$. Since any possible value of $DL$ is yielded by $r^2$ different $D$'s, there are in total $\frac{\abs{\h/\rho\h}}{r^2} = r^2$ elements of the $\rho$-orbit of $L$.

        \item Clearly $0$ belongs to any $\rho$-orbit. If $L_2\equiv DL_1\pmod{\rho}$ and $L_1,L_2\nequiv 0\pmod{\rho}$, then the $\rho$-orbit of $L_2$ must be a subset of the $\rho$-orbit of $L_1$. But by part (i) both of these $\rho$-orbits contain $r^2$ elements, so they must be the same.

        \item From part (i) we know that any $\rho$-orbit contains $r^2-1$ elements different from $0$ whose norm is divisible by $\rho$. Let us now derive the total number of elements $L\in(\h/\rho\h)\setminus\set{0}$ such that $\rho\mid\nmh(L)$. By Lemma~\ref{lem:matrixisomorphism}, we only need to count the non-zero matrices
        \[
            \psi(L) =
            \left(\begin{matrix}
                x&y\\z&w
            \end{matrix}\right)
            \in \matr_2(\ok/\rho\ok)
        \]
        with determinant $0$, i.e.\ $xw-yz\equiv0\pmod{\rho}$. Let us consider several cases.
        \begin{enumerate}
            \item $x\equiv w\equiv 0\pmod{\rho}$. Then one of $y$, $z$ must be zero and the other must be non-zero. This contributes $2(r-1)$ valid $\psi(L)$'s.
            \item Exactly one of $x$, $w$ is non-zero -- there are $2(r-1)$ such pairs. Any $y$, $z$ such that they are not both non-zero will then suffice; there are
            \(
                r^2 - (r-1)^2 = 2r-1
            \)
            such pairs. Altogether, this case contributes
            \(
                2(r-1)(2r-1) = 4r^2-6r+2
            \)
            valid $\psi(L)$'s.
            \item Both $x$, $w$ are non-zero, which means $y$ and $z$ must be non-zero too. The choice of $x$, $y$, $z$ then uniquely determines $w$ as $\inv{x}yz\pmod{\rho}$, so this contributes $(r-1)^3$ valid $\psi(L)$'s.
        \end{enumerate}
        Altogether, we have
        \begin{align*}
            2(r-1) + 4r^2-6r+2 + (r-1)^3 = r^3 + r^2 - r - 1 = (r^2-1)(r+1)
        \end{align*}
        elements $L\in(\h/\rho\h)\setminus\set{0}$ whose norm is divisible by $\rho$. Each of these belongs to exactly one $\rho$-orbit and each $\rho$-orbit contains exactly $r^2-1$ such elements. Thus there must be exactly $r+1$ different $\rho$-orbits.\qedhere
    \end{enumerate}
\end{proof}

\section{Results from geometry of numbers}
\label{sec:gon}
In this section, we will use Minkowski's theorem to prove that all but finitely many (up to multiplication by a unit) primes of $\ok$ satisfy the following condition that will help us in section~\ref{sec:pid} to prove that $\h$ is a principal ideal domain in some cases.

\begin{definition}
    \label{def:suitability}
    We say a prime element $\rho\in \ok$ is \emph{$\h$-suitable}, if for any $L\in\h$ such that $\rho\nmid L$, the fractional ideal $\h\frac{L}{\rho}+\h$ contains a non-zero element $L'$ with $\nn(L')<1$, where $\nn=\nmk\circ\nmh$ is the double norm.
\end{definition}

First, we interpret fractional ideals of $\h$ as lattices. The following notation will be used throughout this section.
\begin{definition}
    Let $L=x+y\alpha+z\beta+w\alpha\beta$, $x,y,z,w\in K$ be a quaternion. Then let us define an injective linear mapping $f$ from the set of such quaternions to $\R^{4n}$ by
    \[
        f(L) := \Big(\sigma_1(x), \sigma_1(y), \sigma_1(z), \sigma_1(w), \sigma_2(x), \sigma_2(y), \sigma_2(z), \sigma_2(w), \dots, \sigma_n(x), \sigma_n(y), \sigma_n(z), \sigma_n(w) \Big).
    \]
    Then we can interpret any (left or right) fractional ideal $\mathcal I$ of $\h$ as a lattice $f(\mathcal{I})$.
\end{definition}

\begin{lemma}
    \label{lem:latticedet}
    The determinant of $f(\h)$ is $d_K^2$, where $d_K$ is the discriminant of $K$. Further, let $\rho\in\ok$ be a prime element that does not divide $T$ and $L\in\h$ a quaternion such that $\rho\mid\nmh(L)$, $\rho\nmid L$. Then
    \begin{equation}
        \label{eq:latticedet}
        \det f\zav{\h\frac{L}{\rho}+\h} = \frac{d_K^2}{\nmk(\rho)^2}.
    \end{equation}
\end{lemma}
\begin{proof}
    Let $\omega_1,\dots,\omega_n$ be an integral basis of $\ok$. Then
    \begin{alignat*}{5}
        &f(\omega_1),\quad            &&f(\omega_2),\quad            &&\dots,\quad  &&f(\omega_n),\\
        &f(\omega_1\alpha),\quad      &&f(\omega_2\alpha),\quad      &&\dots,\quad  &&f(\omega_n\alpha),\\
        &f(\omega_1\beta),\quad       &&f(\omega_2\beta),\quad       &&\dots,\quad  &&f(\omega_n\beta),\\
        &f(\omega_1\alpha\beta),\quad &&f(\omega_2\alpha\beta),\quad &&\dots,\quad  &&f(\omega_n\alpha\beta)
    \end{alignat*}
    is a basis of $f(\h)$ (as a $\Z$-module). This means that
    \[
        \det f(\h) = \zav{\det(\sigma_r(\omega_s))_{r,s=1}^n}^4 = d_k^2.
    \]

    The elements of $\h\frac{L}{\rho}+\h$ are of the form $\frac{DL+C\rho}{\rho}$ for $D,C\in\h$, i.e.\ their numerators attain precisely the values whose residue classes belong to the $\rho$-orbit of $L$. So by Theorem~\ref{lem:orbits}, the index of $\h$ within $\h\frac{L}{\rho}+\h$ is $\abs{\nmk(\rho)}^2 = \nmk(\rho)^2$, which proves \eqref{eq:latticedet}.
\end{proof}

Now we need a centrally symmetric convex set. For this, we generalize an approach used by Deutsch in \cite{deutsch1, deutsch2}. We will use $\vol$ to denote the volume of a set.
\begin{lemma}
    Let us define the \emph{spherical diamond} in $\R^{4n}$ as
    \[
        S_n(r) = \set{(x_1,\dots,x_{4n})\in\R^{4n}: \sum_{t=0}^{n-1} \sqrt{x_{4t+1}^2+x_{4t+2}^2+x_{4t+3}^2+x_{4t+4}^2} < r}.
    \]
    Then this set is centrally symmetric and convex and its volume is $\frac{\pi^{2n}12^n}{(4n)!} r^{4n}$.
\end{lemma}
\begin{proof}
    That $S_n(r)$ is centrally symmetric is obvious. By triangle inequality, we have that
    \[
        \sqrt{(x_1+y_1)^2+(x_2+y_2)^2+(x_3+y_3)^2+(x_4+y_4)^2} \leq \sqrt{x_1^2+x_2^2+x_3^2+x_4^2} + \sqrt{y_1^2+y_2^2+y_3^2+y_4^2},
    \]
    so for any two points $(x_1,\dots,x_{4n}), (y_1,\dots,y_{4n})\in S_n(r)$, their midpoint satisfies
    \begin{multline*}
        \sum_{t=0}^{n-1} \sqrt{\zav{\frac{x_{4t+1}+y_{4t+1}}{2}}^2+\zav{\frac{x_{4t+2}+y_{4t+2}}{2}}^2+\zav{\frac{x_{4t+3}+y_{4t+3}}{2}}^2+\zav{\frac{x_{4t+4}+y_{4t+4}}{2}}^2} \leq{}\\{}\leq \frac12\sum_{t=0}^{n-1} \sqrt{x_{4t+1}^2+x_{4t+2}^2+x_{4t+3}^2+x_{4t+4}^2} + \frac12\sum_{t=0}^{n-1} \sqrt{y_{4t+1}^2+y_{4t+2}^2+y_{4t+3}^2+y_{4t+4}^2} < \frac{r+r}{2} = r.
    \end{multline*}
    Thus, $S_n(r)$ is convex.

    We will prove by induction that $\vol S_n(r) = \frac{\pi^{2n}12^n}{(4n)!} r^{4n}$. This is true for $n=1$, since $S_1(r)$ is just the $4$-dimensional ball, which has volume $\frac{\pi^2}{2}r^4$ and surface area $2\pi^2r^3$. We can then express $\vol S_{n+1}(r)$ in terms of $\vol S_n(r)$ followingly:
    \begin{align*}
        \vol S_{n+1}(r) &= \int_0^r 2\pi^2(r-y)^3 \cdot \vol S_n(y) \dd y =\\&= 2\pi^2\cdot \frac{\pi^{2n}12^n}{(4n)!} \int_0^r r^3y^{4n} - 3r^2y^{4n+1} + 3ry^{4n+2} - y^{4n+3} \dd y =\\&= r^{4n+4} \cdot \frac{\pi^{2n}12^n}{(4n)!} \cdot 2\pi^2 \cdot \zav{\frac1{4n+1}-\frac3{4n+2}+\frac3{4n+3}-\frac1{4n+4}} =\\&= r^{4n+4} \cdot \frac{\pi^{2n}12^n}{(4n)!} \cdot 2\pi^2 \cdot \frac{6}{(4n+1)(4n+2)(4n+3)(4n+4)} = \frac{\pi^{2(n+1)}12^{n+1}}{(4(n+1))!}r^{4(n+1)}.\qedhere
    \end{align*}
\end{proof}
\begin{corollary}
    The volume of the set
    \[
        J_n(r) = \set{(x_1,\dots,x_{4n})\in\R^{4n}: \sum_{t=0}^{n-1} \sqrt{(\sigma_t(Q))\Big(x_{4t+1},x_{4t+2},x_{4t+3},x_{4t+4}^2\Big)} < r}
    \]
    is $\frac{\pi^{2n}48^n}{(4n)!\nmk(T)}r^{4n}$.
\end{corollary}
\begin{proof}
    Let $M=(m_{rs})_{r,s=1}^4$ be the $4\times4$ matrix on the right hand side in \eqref{eq:to4squares}. Let $\sigma_t(M)$ denote the matrix where $\sigma_t$ is applied element-wise, i.e.\ $(\sigma_t(m_{rs}))_{r,s=1}^4$. Then because of \eqref{eq:to4squares}, the linear transformation $\R^{4n}\to\R^{4n}$ defined by multiplication by a block diagonal matrix with diagonal blocks $\sigma_1(M), \sigma_2(M), \dots, \sigma_n(M)$ maps $J_n(r)$ bijectively onto $S_n(r)$. Since $\det M = \frac{T}{4}$, the determinant of this transformation is $\frac{\nmk(T)}{4^n}$. This means that
    \[
        \vol J_n(r) = \frac{4^n}{\nmk(T)}\cdot \vol S_n(r) = \frac{\pi^{2n}48^n}{(4n)!\nmk(T)}r^{4n}.\qedhere
    \]
\end{proof}

\begin{theorem}
    \label{thrm:suitabilitybound}
    Let $\rho\in\ok$ be a prime element that does not divide $T$. Then if
    \begin{equation}
        \label{eq:suitabilitybound}
        \abs{\nmk(\rho)} > \frac{\sqrt{(4n)!}}{\pi^n3^{n/2}n^{2n}}\cdot d_K\sqrt{\nmk(T)},
    \end{equation}
    then $\rho$ is $\h$-suitable.
\end{theorem}
\begin{proof}
    If $\rho\nmid\nmh(L)$, then there is a $\lambda\in\ok$ such that $\lambda\cdot\nmh(L)\equiv 1\pmod{\rho}$, meaning there is a $C\in\h$ such that $\lambda\nmh(L)-C\rho=1$. Then we simply set $L':=\lambda\overline{L}\cdot\frac{L}{\rho}-C=\frac{1}{\rho}$, which clearly suffices.

    From now, we may presume $\rho\mid\nmh(L)$. Consider the lattice $f\zav{\h\frac{L}{\rho}+\h}$ and the set $J_n(n)$. By Lemma~\ref{lem:latticedet} and the condition on $\nmk(\rho)$, we have
    \begin{align*}
        2^{4n}\cdot\det f\zav{\h\frac{L}{\rho}+\h} &= \frac{16^{n}d_K^2}{\nmk(\rho)^2} < 16^{n}d_K^2 \cdot \frac{\pi^{2n}3^{n}n^{4n}}{(4n)!d_K^2\nmk(T)} = \frac{\pi^{2n}48^{n}}{(4n)!\nmk(T)}\cdot n^{4n} = \vol J_n(n),
    \end{align*}
    so the conditions of Minkowski's theorem are met. By this theorem, there is a non-zero element $L'\in \h\frac{L}{\rho}+\h$ such that $f(L')\in J_n(n)$. Since the coordinates of $f(L')$ are the coefficients of $L'$ in the embeddings $\sigma_1,\dots,\sigma_n$, the fact that $f(L')\in J_n(n)$ together with
    \[
        (\sigma_t(Q))\Big(\sigma_t(x),\sigma_t(y),\sigma_t(z),\sigma_t(w)\Big) = \sigma_t\Big(Q(x,y,z,w)\Big) = \sigma_t\Big(\nmh(x+y\alpha+z\beta+w\alpha\beta)\Big)
    \]
    means that $\sum_{t=1}^n \sqrt{\sigma_t(\nmh(L'))} < n$.
    From this, the inequality between the arithmetic and geometric means gives us
    \begin{align*}
        \nn(L') = \nmk(\nmh(L')) &= \zav{\sqrt[n]{\prod_{t=1}^n \sqrt{\sigma_t(\nmh(L'))}}}^{2n} \leq \zav{\frac1n \sum_{t=1}^n \sqrt{\sigma_t(\nmh(L'))}}^{2n} < \zav{\frac nn}^{2n} = 1,
    \end{align*}
    just as we wished to prove.
\end{proof}

\section{When $\h$ is a principal ideal domain}
\label{sec:pid}
As mentioned previously, we shall prove that if every prime element of $\ok$ is $\h$-suitable, then $\h$ is a principal ideal domain. Then we shall use this to prove with some additional conditions that $Q$ represents all elements of $\ok$

\begin{theorem}
    \label{thrm:suitability-pid}
    If every prime element $\ok$ is $\h$-suitable, then $\h$ is a principal ideal domain.
\end{theorem}
\begin{proof}
    Let $\mathcal{I}$ be a left-sided ideal of $\h$. Since $\nn$ maps non-zero elements of $\h$ to positive integers, we may choose a $D\in \mathcal{I}\setminus\set{0}$ such that $\nn(G)$ is minimal. Suppose for contradiction that some $C\in \mathcal{I}$ is not contained in $\h D$. We will show that this would mean the existence of an element $D'\in \mathcal{I}\setminus\set{0}$ with $\nn(D')<\nn(D)$, forming a contradiction.

    Let $CD^{-1}=\frac{C\overline{D}}{\nmh(D)}=\frac{L}{\lambda}$ for some $L\in\h$, $\lambda\in\ok$. Then $\lambda$ cannot be a unit, since $C\notin \h D$, so there is a prime element $\rho\in\ok$ that divides $\lambda$. From Definition~\ref{def:suitability}, there exists a non-zero element $L'\in\h\frac{L}{\rho}+\h \subseteq \h\frac{L}{\lambda}+\h$ that satisfies $\nn(L')<1$. This means that a non-zero element $L'D\in \h C + \h D \subset \mathcal{I}$ satisfies $\nn(L'D)<\nn(D)$, which is a contradiction. This means that $\mathcal{I}=\h D$.
\end{proof}

With this, it is easy to show that a $\rho$-orbit can be represented by an element of $\h$ whose norm is $\rho$ up to a multiplication by a unit.

\begin{lemma}
    \label{lem:orbithensel}
    Let $\rho\in\ok$ be a prime element that does not divide $T$. Then if $L\in\h$ such that $\rho\nmid L$ but $\rho\mid\nmh(L)$, there is a $C\in\h$ such that $\nmh(L+C\rho)$ is divisible by $\rho$ but not by $\rho^2$.
\end{lemma}
\begin{proof}
    Let $L=x+y\alpha+z\beta+w\alpha\beta$. If the derivative of $Q(x,y,z,w)$ with respect to any of the four variables $x$, $y$, $z$, $w$ were non-zero, we could apply Hensel's lemma to the polynomial $Q(x,y,z,w)-\rho$ to show that $\nmh(L+C\rho)\equiv \rho\pmod{\rho^2}$ for some $C\in\h$. Thus, it suffices to show that all four of these derivatives being zero leads to a contradiction. Indeed, this would yield the system of congruences
    \begin{align*}
        2x + \mu y - \nu w &\equiv 0\pmod{\rho},\\
        \mu x + 2Ay + \nu z &\equiv 0\pmod{\rho},\\
        \nu y + 2Bz + B\mu w &\equiv 0\pmod{\rho},\\
        -\nu x + B\mu z + 2ABw &\equiv 0\pmod{\rho}.
    \end{align*}
    But the determinant of this system is $T^2\nequiv 0\pmod{\rho}$, so the only solution is $x\equiv y\equiv z\equiv w\equiv 0\pmod{\rho}$, which contradicts $\rho\nmid L$.
\end{proof}

\begin{lemma}
    \label{lem:quaterniongcd}
    Let $\h$ be a principal ideal domain, $\rho\in\ok$ a prime element and $L\in\h$ such that $\rho\mid\nmh(L)$ but $\rho^2\nmid\nmh(L)$. Then the ideal $\h\rho + \h L$ contains an element $L_0$ such that $\nmh(L_0)=u\rho$ and $u\in\ok$ is a unit.
\end{lemma}
\begin{proof}
    Since $\h$ is a principal ideal domain, clearly $\h\rho + \h L = \h L_0$ for some $L_0$. Then both $\rho$ and $L$ are left-sided multiples of $L_0$, so $\nmh(L_0)$ is a common divisor of $\nmh(\rho)=\rho^2$ and $\nmh(L)$, which means $\nmh(L_0)\mid \rho$. On the other hand, we have $L_0 = C\rho+DL$ for some $C,D\in\h$, which means that
    \begin{align*}
        \nmh(L_0) &= (C\rho+DL)\overline{(C\rho+DL)} = \nmh(C)\rho^2 + C\rho \overline{(DL)} + DL\overline{(C\rho)} + \nmh(DL) =\\&= \rho^2\nmh(C) + \rho\zav{DL\overline{C}+\overline{(DL\overline{C})}} + \nmh(L)\nmh(D) \equiv 0\pmod{\rho},
    \end{align*}
    since $DL\overline{C}+\overline{(DL\overline{C})} = \nmh\zav{DL\overline{C}+1} - \nmh\zav{DL\overline{C}} - 1$ is an element of $\ok$. So $\rho\mid\nmh(L_0)$, which means that $\nmh(L_0)=u\rho$ for some unit $u\in\ok$.
\end{proof}

Now for a totally positive prime element $\rho$, in order to obtain a representation of $\rho$ from that of $u\rho$, we simple constrain ourselves to fields for which every totally positive unit in $\ok$ is a square -- it is well known that this is equivalent to the existence of units of every signature \cite[p. 111, Corollary 3]{narkiewicz}. 
Lemmas~\ref{lem:quaterniongcd} and \ref{lem:orbithensel} then imply that $Q$ represents any totally positive prime element $\rho\in\ok^+$ that does not divide $T$, since there exists at least one $\rho$-orbit by Lemma~\ref{lem:orbits}.
It remains to deal with prime elements $\rho\mid T$. For this, notice that $\nmh(\nu-\mu\beta+2\alpha\beta) = T$, so if $\rho^2\nmid T$, we have a suitable $L$. This gives us the following corollary.
\begin{corollary}
    \label{cor:final}
    If $\h$ is a principal ideal domain, every totally positive unit in $\ok$ is a square and $T$ is a squarefree element\footnote{An element of $\ok$ is said to be \emph{squarefree} if it is not divisible by any square of a non-unit.} in $\ok$ then $Q$ is universal over $K$.
\end{corollary}
\begin{proof}
    Let us consider any prime element $\rho\in\ok^+$. If $\rho\nmid T$, then there exists a $\rho$-orbit and thus an $L_1\in\h$ such that $\rho\nmid L_1$ and $\rho\mid\nmh(L_1)$, so by Lemma~\ref{lem:orbithensel} there is also an $L\in\h$ such that $\nmh(L)$ is divisible by $\rho$ but not by $\rho^2$. If $\rho\mid T$, then we simply take $L=\nu-\mu\beta+2\alpha\beta$, since $T$ is squarefree.

    Now by Lemma~\ref{lem:quaterniongcd}, there exists an $L_0$ such that $\nmh(L_0)=u\rho$, where $u$ is a unit, for any totally positive prime element $\rho\in\ok$. Clearly $\nmh(L_0)$ is totally positive, so $u$ is also totally positive and thus a square, i.e.\ $u=v^2$ for a unit $v$. Then $\nmh(v^{-1}L_0) = \rho$, so $Q$ represents any totally positive prime element.

    Finally, recall that every totally positive unit in $\ok$ is a square if and only if $\ok$ contains units of every signature. Thus any prime element $\rho\in\ok$ has a multiple $u\rho\in\ok^+$ where $u$ is a unit, so any element of $\ok^+$ can be written as a product of totally positive prime elements. Multiplicativity of $\nmh$ then means that $Q$ represents every element of $\ok^+$, as we wished to prove.
\end{proof}

\section{The form $x^2+xy+y^2+yz-xw+z^2+zw+w^2$ over $\Q(\zeta_7+\zeta_7^{-1})$}
\label{sec:2form}

Now, we apply our findings to a particular example. We will work with the cubic field $K=\Q(\zeta_7+\zeta_7^{-1})$, where $\zeta_7$ a primitive seventh root of unity. The minimal polynomial of $\zeta_7+\zeta_7^{-1}$ is $x^3+x^2-2x-1$ and $K$ has discriminant $d_k=49$. It is also well known that $\ok$ is a unique factorization domain and that $\ok$ contains units of every signature, and thus every totally positive unit in $\ok$ is a square. We shall use $\phi_m = \zeta_7^m+\zeta_7^{-m}$, $m=1,2,3$ as an integral basis of $\ok$ to make the following useful observation.
\begin{lemma}
    \label{lem:norm-mod-7}
    For any $\lambda\in\ok$, the integer $\nmk(\lambda)$ can only attain residues $0$, $1$ or $-1$ mod $7$.
\end{lemma}
\begin{proof}
    If $\lambda = x\phi_1+y\phi_2+z\phi_3$, $x,y,z\in\Z$, then we can directly express
    \begin{align*}
        \nmk(\lambda) &= x^3+y^3+z^3+x^2y+3y^2z+3z^2x-4x^2z-4y^2x-4z^2y-xyz =\\
        &= (x+y+z)^3 - 7(x^2z+y^2x+z^2y+xyz).
    \end{align*}
    We finish by observing that $0$, $1$ and $-1$ are the only cubic residues mod $7$.
\end{proof}

We will now examine the quadratic form and corresponding quaternionic ring determined by $(A,B,\mu,\nu)=(1,1,1,1)$. This means that $T=2$ and
\[
Q(x,y,z,w) = x^2+xy+y^2+yz-xw+z^2+zw+w^2,
\]
which is equivalent to \eqref{eq:hurwitz-form}. Now if $2$ were reducible in $\ok$, there would be a prime element with norm $\pm2$, which would contradict Lemma~\ref{lem:norm-mod-7}, so $T=2$ is a prime element in $\ok$.

We wish to prove that $Q$ is universal over $K$. If we use Corollary~\ref{cor:final}, it remains only to show that $\h$ is a principal ideal domain. Theorem~\ref{thrm:suitabilitybound} tells us that for that, we need to just check $\h$-suitability of prime elements that divide $T$ (i.e.\ just $2$ itself) and of those $\rho$ whose $\abs{\nmk(\rho)}$ is at most
\[
    \frac{\sqrt{12!}}{\pi^3\cdot3^{3/2}\cdot3^6}\cdot 49\sqrt{\nmk(2)} \approx 25.83.
\]
Using Lemma~\ref{lem:norm-mod-7}, we see that the only possible norms of these primes are $7$, $8$ and $13$, since they are the only prime powers $\leq 25$ that are congruent to $1$, $0$ or $-1$ mod $7$. For these norms, we have prime elements
\begin{align*}
    \rho_7 &= 2-\phi_1, & \rho_8 &= 2, & \rho_{13} &= 3+\phi_1.
\end{align*}
Each $\rho_p$ is the only prime of its norm up to automorphism and multiplication by a unit. Since automorphisms and multiplication by units clearly preserve $\h$-suitability, we only need to check that these three prime elements are $\h$-suitable.

As we've seen in the proof of Theorem~\ref{thrm:suitabilitybound}, in showing the existence of a non-zero $L'\in\h\frac{L}{\rho}+\h$ with $\nn(L')<1$, we only need to deal with the case $\rho\mid\nmh(L)$. We will thus show the $\h$-suitability of each $\rho_t$, $t=7,8,13$ by showing that each of their $\rho_t$-orbits can be represented by an element whose norm is $\rho_t$. This will mean that any $\h\frac{L}{\rho}+\h$ contains an $L'$ such that $\nmh(L')=\frac1\rho$, meaning $\nn(L')=\frac1{\nmk(\rho)}<1$. We deal with the three prime elements separately:
\begin{enumerate}
    \item $\rho=\rho_8=2$. We will show that there is only one $2$-orbit. Suppose that $L=x+y\alpha+z\beta+w\alpha\beta$ and $2\mid \nmh(L)$. This means that
    \begin{align*}
        0 &\equiv \nmh(L) = x^2+xy+y^2 + yz-xw + z^2+zw+w^2 \equiv\\&\equiv x^2+z^2 + xy+xw+zy+zw + y^2+w^2 \equiv\\&\equiv (x+z)^2+(x+z)(y+w)+(y+w)^2\pmod2.
    \end{align*}
    Since the polynomial $t^2+t+1$ has no roots in the finite field $\ok/2\ok$, the only solution to $a^2+ab+b^2\equiv 0\pmod 2$ is $a\equiv b\equiv 0$. Thus $2$ divides both $x+z$ and $y+w$, so
    \[
        L = \left(\frac{x+z}2 + \frac{y+w}2\alpha + \frac{z-x}2\beta + \frac{w-y}2\alpha\beta\right)\cdot(1+\beta),
    \]
    meaning the only $2$-orbit is that of $1+\beta$. But the (quaternionic) norm of $1+\beta$ is $2$, meaning $2$ is $\h$-suitable.

    \item $\rho=\rho_7=2-\phi_1$. By Lemma~\ref{lem:orbits}, there are $8$ left $\rho_7$-orbits in total. Using the fact that $\phi_2^2+\phi_2+1=\rho_7$, they can be represented by the following elements with quaternionic norm $\rho_7$:
    \begin{align*}
        & \phi_2+\alpha, && \phi_2\alpha+\beta, && \phi_2\beta+\alpha\beta, && \phi_2-\alpha\beta,\\
        & 1+\phi_2\alpha, && \alpha+\phi_2\beta, && \beta+\phi_2\alpha\beta, && 1-\phi_2\alpha\beta.
    \end{align*}

    \item $\rho=\rho_13 = 3+\phi_1$. By Lemma~\ref{lem:orbits}, there are $14$ left $\rho_{13}$-orbits in total. Using the fact that
    \[
        \rho_{13} = 3+\phi_1 = \phi_2^2 + \phi_2\phi_3 + \phi_3^2 = 1 + \phi_3^2 = \phi_3^2 - \phi_3(1+\phi_3) + (1+\phi_3)^2 - (1+\phi_3) + 1,
    \]
    these $14$ orbits can be represented by the following elements with quaternionic norm $\rho_{13}$:
    \begin{align*}
        & \phi_3+\phi_2\alpha, && \phi_3\beta+\phi_2\alpha\beta, && \phi_3-(1+\phi_3)\alpha+\beta, && \phi_3+\beta,\\
        & \phi_2+\phi_3\alpha, && \phi_2\beta+\phi_3\alpha\beta, && 1-(1+\phi_3)\alpha+\phi_3\beta, && 1+\phi_3\beta,\\
        & \phi_3\alpha+\phi_2\beta, && \phi_3-\phi_2\alpha\beta, && \phi_3\alpha-(1+\phi_3)\beta+\alpha\beta,\\
        & \phi_2\alpha+\phi_3\beta, && \phi_2-\phi_3\alpha\beta, && \alpha-(1+\phi_3)\beta+\phi_3\alpha\beta.
    \end{align*}
\end{enumerate}
Thus by Theorem~\ref{thrm:suitability-pid}, $\h$ is a principal ideal domain, meaning that Corollary~\ref{cor:final} gives us:
\begin{theorem}
    \label{thrm:2form-universal}
    The quadratic form
    \(
        x^2+xy+y^2+yz-xw+z^2+zw+w^2
    \)
    is universal over $\Q(\zeta_7+\zeta_7^{-1})$.
\end{theorem}
Note that this quadratic is equivalent to the form~\eqref{eq:hurwitz-form}, since substituting
\begin{align*}
    X &= 2x+y-w, & Y &= -x,\\
    Z &= -x+z+w, & W &= -x+w
\end{align*}
yields $X^2+Y^2+Z^2+W^2+XY+XZ+XW = x^2+xy+y^2+yz-xw+z^2+zw+w^2$. This means that by Theorem~\ref{thrm:2form-universal}, the form \eqref{eq:hurwitz-form} is also universal over $\Q(\zeta_7+\zeta_7^{-1})$.

Let us also briefly note that when $\h$ is a principal ideal domain, it is also possible to prove an analogue of Jacobi's four-square theorem, that is, to derive an explicit formula, which given a totally positive element $\lambda\in\ok^+$ produces the number of quaternions $L\in\h$ with $\nmh(L)=\lambda$. Let us denote this number $r_{\h}(\lambda)$. This formula will typically involve a sum over ideals of $\ok$ that divide $\lambda\ok$ (with some restriction related to $T$) multiplied by the number of elements in $\h$ with quaternionic norm $1$. For $K=\Q(\zeta_7+\zeta_7^{-1})$ and $(A,B,\mu,\nu)=(1,1,1,1)$, there are exactly $24$ elements of norm $1$ in $\h$, so the number of representations is
\begin{equation}
    \label{eq:jacobi}
    r_{\h}(\lambda) = 24\sum_{\substack{\lambda\ok\subseteq\delta\ok\\2\nmid\delta}} \abs{\ok/\delta\ok},
\end{equation}
where the sum is over ideals $\delta\ok$ of $\ok$. Once it is known that $\h$ is a principal ideal domain, the proof of the representation formula can be carried out for example by modifying the proof of classical Jacobi's theorem due to Hurwitz \cite{hurwitz} in a straightforward manner.

\section{The form $x^2+xy+y^2+z^2+zw+w^2$ over $\Q(\zeta_7+\zeta_7^{-1})$}
\label{sec:3form}

Let us examine another quadratic form $Q=x^2+xy+y^2+z^2+zw+w^2$ over $K=\Q(\zeta_7+\zeta_7^{-1})$. In the construction of Definition~\ref{def:H}, this corresponds to $(A,B,\mu,\nu)=(1,1,1,0)$ and results in $T=3$, which is a prime element of $\ok$, since by Lemma~\ref{lem:norm-mod-7} there cannot be a prime element with norm $3$. The approach of the previous section fails for this form, since the quaternionic ring $\h$ is known to have class number $2$ \cite[Table 8.3]{kirschmer-voight}. The issue is that in Lemma~\ref{lem:quaterniongcd}, the ideal $\h\rho+\h L$ may be non-principal. However, the fact that the class number of $\h$ is $2$ can be used to obtain some weaker results on the representation of certain special elements of $\ok$ by $Q$.

We shall be using the fact that the determinant of multiplication by a quaternion $L$ viewed as a linear transformation of a vector space of quaternions over $K$ is $\nmh(L)^2$. If we then consider an ideal $\mathcal I$ and its associated lattice $f(\mathcal{I})$ as defined in Section~\ref{sec:gon}, we get $\det f(\mathcal{I}\cdot L) = \det f(\mathcal{I})\cdot \nn(L)^2$.

Let $\rho_7 = 2 - (\zeta_7+\zeta_7^{-1})$, this is a prime element of norm $7$ in $\ok$. It will also be useful that $\rho_7$ is the only prime element of norm $7$ up to multiplication by a unit, since it divides its conjugates $2 - (\zeta_7^2+\zeta_7^{-2})$ and $2 - (\zeta_7^3+\zeta_7^{-3})$. We start by fixing the ideal $\mathcal{S}:= \h(1+\alpha+2\beta)+\h\rho_7$. This ideal is non-principal and thus, since the class number of $\h$ is $2$, it is equivalent to any other non-principal ideal of $\h$. By the results of Section~\ref{sec:gon}, the determinant of the lattice associated to $\frac1{\rho_7}\mathcal S$ is $\det f(\frac1{\rho_7}\mathcal{S})=\frac{d_K^2}{\nmk(\rho_7)^2}$, which then implies $\det f(\mathcal{S}) = \frac{d_K^2}{\nmk(\rho_7)^2}\cdot\nn(\rho_7)^2 = d_K^2\cdot 7^2$. We can also observe that $\rho_7\mid\nmh(P)$ for any $P\in\mathcal{S}$; this is derived in the same way as in the proof of Lemma~\ref{lem:quaterniongcd}.

\begin{lemma}
    \label{lem:squarefree-ideals}
    Let $\lambda\in\ok^+$ be squarefree. Then there exists a quaternion $L\in\h$ such that the ideal $\mathcal{I}:=\h L + \h\lambda$ has $\det f(\mathcal{I}) = d_K^2\cdot\nmk(\lambda)^2$.
\end{lemma}
\begin{proof}
    Firstly, if $\lambda$ is a prime element not dividing $T$, the existence of $\lambda$-orbits guarantees the existence of such $L$. The determinant comes from
    \[
        \det f(\mathcal{I}) = \det f\zav{\h\frac L\lambda + \h} \cdot \nn(\lambda)^2 = \frac{d_K^2}{\nmk(\lambda)^2}\cdot\nmk(\lambda)^4.
    \]
    If $\lambda$ is a prime element that divides $T=3$, that it is $3$ itself up to a multiplication by a unit. Then $L=1+\alpha$ suffices, since $\nmh(L)=3$, so
    \[
        \det f(\mathcal{I}) = \det f(\h L) = \det f(\h)\cdot \nn(L)^2 = d_K^2\cdot\nmk(3)^2.
    \]

    If $\lambda$ is not a prime element, we use Chinese remainder theorem: for each prime element $\rho\mid\lambda$, there is an ideal $\h P + \h\rho$. By Chinese remainder theorem, since all of these $\rho$ are distinct and thus coprime, we can find an $L$ modulo $\lambda$ that corresponds to these $P$ modulo $\rho$. Clearly, this will give a relation on the indices
    \[
        \abs{\h / (\h L+\h\lambda)} = \prod_{\rho} \abs{\h / (\h P+\h\rho)} = \prod_{\rho} \nmk(\rho)^2 = \nmk(\lambda)^2,
    \]
    leading to $\det f(\mathcal{I}) = \det f(\h L+\h\lambda) = d_K^2\cdot\nmk(\lambda)^2$.
\end{proof}

\begin{theorem}
    Let $\theta\in\ok^+$ be such that there is a $P\in \mathcal{S}$ with $\nmh(P)=\rho_7\cdot\theta$ and $Q$ also represents $\theta$. Then $Q$ represents $\lambda\theta$ for any $\lambda\in\ok^+$.
\end{theorem}
\begin{proof}
    It suffices to prove the theorem just for squarefree elements $\lambda$. Let $\mathcal{I} = \h L + \h\lambda$ be the ideal from Lemma~\ref{lem:squarefree-ideals}. If $\mathcal I$ is principal, then it is generated some quaternion with norm $\lambda$. Since we are also presuming that $Q$ represents $\theta$, i.e. there is a quaternion with norm $\theta$, the proof in this case is concluded by the multiplicativity of the norm.

    Thus from now on, let us presume that $\mathcal I$ is non-principal. Then it is equivalent to $\mathcal S$, so $\mathcal{I} = U\cdot\mathcal{S}$ for some quaternion $U$ over $K$. By comparing the determinants of lattices, we get
    \[
        \nn(U)^2 = \frac{\det f(\mathcal{I})}{\det f(\mathcal{S})} = \frac{\nmk(\lambda)^2}{7^2},
    \]
    so $\nn(U)=\frac{\nmk(\lambda)}{7}$. This means that $\nmh(U) = \frac{\gamma}{\delta}\in K$ for some $\gamma,\delta\in\ok$ with $\nmk(\gamma)=\nmk(\lambda)$ and $\nmk(\delta)=7$. Since $\rho_7$ is the only prime element of norm $7$ up to multiplication by a unit, we can assume $\delta=\rho_7$.
    Further, because $\lambda\in\mathcal{I}$, we have $\lambda = U\cdot P_0$ for some $P_0\in \mathcal{S}$, whose norm we know to be $\rho_7\cdot\eta$ for some $\eta\in\ok^+$. This leads to
    \[
        \lambda^2 = \nmh(\lambda) = \frac{\gamma}{\delta}\cdot\rho_7\eta = \gamma\eta.
    \]
    Now $\gamma$ divides $\lambda$ and has the same norm, so it is the same up to multiplication by a unit. Since multiplication by a unit of $\ok$ preserves the ideal $\mathcal I$, we can assume $\nmh(U) = \frac{\lambda}{\rho_7}$.

    Now, since we are presuming the existence of $P\in\mathcal{S}$ such that $\nmh(P)=\rho_7\theta$, we have $U\cdot P \in \mathcal{I}\subseteq \h$ and $\nmh(UP) = \frac{\lambda}{\rho_7}\cdot\rho_7\theta = \lambda\theta$, as we wished. Thus, independently of whether $\mathcal I$ is principal or not, we find a quaternion with norm $\lambda\theta$.
\end{proof}

Now it remains to find some elements $\theta$. These seem to be abundant, and in fact, we have not been able to find a totally positive prime element $\theta$ that does not meet the conditions of the theorem. This is somewhat consistent with the conjecture that $Q$ is in fact universal over $K$. For example, it can be verified that the values $\theta\in\set{2-(\zeta_7+\zeta_7^{-1}), 2, 3+(\zeta_7+\zeta_7^{-1}), 3}$ satisfy the conditions of the theorem, giving the following corollary.
\begin{corollary}
    The quadratic form $x^2+xy+y^2+z^2+zw+w^2$ over $\Q(\zeta_7+\zeta_7^{-1})$ represents any totally positive multiple of $2-(\zeta_7+\zeta_7^{-1})$, $2$, $3+(\zeta_7+\zeta_7^{-1})$, or $3$. The norms of these four elements are $7$, $8$, $13$ and $27$ respectively.
\end{corollary}

Let us conclude by noting that while the theorem may potentially allow one to prove that $Q$ represents a great number of elements of $\ok^+$, its limitation is that it cannot prove the representation of prime elements. For that, one would have to show that for a given prime $\lambda$, at least one of the ideals $\mathcal{I}$ given by Lemma~\ref{lem:squarefree-ideals} is principal, because the non-principal ones only give the representation of some multiples of $\lambda$.

\section*{Acknowledgment}
The author wishes to thank V\'it\v ezslav Kala for his useful advice and help.


\begin{thebibliography}{99}
    \bibitem[BK1]{blomer-kala1} V. Blomer, V. Kala, \textit{Number fields without $n$-ary universal quadratic forms}, Math. Proc. Cambridge Philos. Soc. 159 (2015), 239--252.
    \bibitem[BK2]{blomer-kala2} V. Blomer, V. Kala, \emph{On the rank of universal quadratic forms over real quadratic fields}, Doc. Math. 23 (2018), 15--34.
    \bibitem[CL+]{cech-etal} M. \v{C}ech, D. Lachman, J. Svoboda, M. Tinkov\'a, K. Zemkov\'a, \textit{Universal quadratic forms and indecomposables over biquadratic fields}, Math. Nachr. 292 (2019), 540--555.
    \bibitem[CKR]{chan-kim-raghavan} W. K. Chan, M.-H. Kim, S. Raghavan, \emph{Ternary universal integral quadratic forms}, Japan. J. Math. 22 (1996), 263--273.
    \bibitem[De1]{deutsch1} J. I. Deutsch, \textit{Geometry of numbers proof of Götzky's four-squares theorem}, J. Number Theory 96 (2002), no. 2, 417--431.
    \bibitem[De2]{deutsch2} J. I. Deutsch, \textit{An alternate proof of Cohn's four squares theorem}, J. Number Theory 104 (2004), no. 2, 263--278.
    \bibitem[De3]{deutsch3} J. I. Deutsch, \textit{A Quaternionic Proof of the Representation Formula of a Quaternary Quadratic Form}, J. Number Theory 113 (2005), no. 1, 149--174.
    \bibitem[De4]{deutsch4} J. I. Deutsch, \textit{A Quaternionic Proof of the Universality of Some Quadratic Forms}, Integers 8(2) (2008), \#A3
    \bibitem[De5]{deutsch5} J. I. Deutsch, \textit{Short proofs of the universality of certain diagonal quadratic forms}, Archiv der Mathematik 91 (2008), no. 1, 44--48.
    \bibitem[De6]{deutsch6} J. I. Deutsch, \textit{Universality of a non-classical integral quadratic form over $\Q(\sqrt5)$}, Acta Arith. 136 (2009), no. 3, 229--242.
    \bibitem[De7]{deutsch7} J. I. Deutsch, \textit{A non-classical quadratic form of Hessian discriminant 4 is universal over $\Q(\sqrt5)$}, Integers 16 (2016), \#A19
    \bibitem[EK]{earnest-khosravani} A. G. Earnest, A. Khosravani, \emph{Universal positive quaternary quadratic lattices over totally real number fields}, Mathematika 44 (1997), 342--347.
    \bibitem[Hu]{hurwitz} A. Hurwitz, \textit{Ueber die Zahlentheorie der Quaternionen}. Nachrichten von der Gesellschaft der Wissenschaften zu Göttingen, Matematisch-physikalische Klasse (1896), 4, 313--340.
    \bibitem[Ka]{kala} V. Kala, \textit{Universal quadratic forms and elements of small norm in real quadratic fields}, Bull. Aust. Math. Soc. 94  (2016), 7--14.
    \bibitem[Ki]{BMkim} B. M. Kim, \emph{Finiteness of real quadratic fields which admit positive integral diagonal septenary universal forms}, Manuscr. Math. 99 (1999), 181--184.
    \bibitem[Km]{MHkim} M.-H. Kim, \emph{Recent developments on universal forms. In Algebraic and arithmetic theory of quadratic forms}, In Algebraic and arithmetic theory of quadratic forms, Contemp. Math. 344,  Amer. Math. Soc., Providence, RI, 2004, 215--228.
    \bibitem[KS]{kala-svoboda} V. Kala, J. Svoboda, \emph{Universal quadratic forms over multiquadratic fields}, Ramanujan J. 48 (2019), 151--157.
    \bibitem[KTZ]{krasensky-tinkova-zemkova} J. Kr\'asensk\'y, M. Tinkov\'a, K. Zemkov\'a, \textit{There are no universal ternary quadratic forms over biquadratic fields}, Proc. Edinb. Math. Soc., 63 (2020), 861-912.
    \bibitem[KV]{kirschmer-voight} M. Kirschmer, J. Voight, \textit{Algorithmic enumeration of ideal classes for quaternion orders}, SIAM J. Comput. 39 (2010), no. 5, 1714--1747.
    \bibitem[KY]{kala-yatsyna} V. Kala, P. Yatsyna, \textit{Lifting problem for universal quadratic forms}, Adv. Math. 377 (2021), 107497, 24 pp.
    \bibitem[Na]{narkiewicz} W. Narkiewicz, \emph{Elementary and analytic theory of algebraic numbers}, 3rd Edition, Springer-Verlag, Berlin, 2004
    \bibitem[Ya]{yatsyna} P. Yatsyna, \emph{A lower bound for the rank of a universal quadratic form with integer coefficients in a totally real field}, Comment. Math. Helvet. 94 (2019), 221--239.

\end{thebibliography}
\end{document}